\documentclass[11pt]{amsart}

\usepackage[T1]{fontenc}
\usepackage{lmodern}
\usepackage{amsmath,amssymb,amsfonts,amsthm}
\usepackage{graphicx}
\usepackage{microtype}
\usepackage{needspace}
\usepackage[hidelinks]{hyperref}
\hypersetup{
  pdftitle={A Universal Obstruction to the Samuelson Condition for Tangent Lagrangian 2-Webs},
  pdfauthor={Yasuhiro Kurokawa},
  pdfsubject={Revised preprint with a smooth local proof for separated families}
}

\newcommand{\R}{\mathbb{R}}
\newcommand{\duv}{\frac{\partial^2}{\partial u\partial v}}

\theoremstyle{plain}
\newtheorem{theorem}{Theorem}[section]
\newtheorem{proposition}[theorem]{Proposition}
\newtheorem{lemma}[theorem]{Lemma}
\newtheorem{corollary}[theorem]{Corollary}

\theoremstyle{remark}
\newtheorem{remark}[theorem]{Remark}

\title[A universal obstruction to the Samuelson condition]{A Universal Obstruction to the Samuelson Condition for Tangent Lagrangian 2-Webs}

\author{Yasuhiro Kurokawa}
\address{Department of Architecture, School of Architecture, Shibaura Institute of Technology, 3-7-5 Toyosu, Koto-ku, Tokyo 135-8548, Japan}
\email{kurokawa@sic.shibaura-it.ac.jp}
\date{Revised preprint, November 4, 2026}

\begin{document}

\begin{abstract}
We study the Samuelson area condition for planar Lagrangian \(2\)-webs
formed by tangent lines. For the tangent line family
\(L_t:y=tx+h(t)\), where \(h\) is smooth and \(h''\neq0\), let
\(\Psi(u,v)=L_u\cap L_v\) for \(u\neq v\), and let \(J_\Psi\) be
its Jacobian. Domitrz and Zubilewicz (2024) proved that these webs
do not satisfy the Samuelson condition. We study the Jacobian as
the two tangents coalesce. In web coordinates, the Samuelson condition
is equivalent to \(\partial_u\partial_v\log|J_\Psi|=0\).
We show that \(J_\Psi\) has a simple zero along \(u=v\) and that
\[
\partial_u\partial_v\log|J_\Psi(u,v)|
=\frac{1}{(v-u)^2}+R_h(u,v),
\]
where \(R_h\) extends smoothly across the diagonal.
The double pole therefore prevents the Samuelson condition from
holding near the envelope. We also give a direct local proof of
this failure for webs formed by two separated smooth line families
with varying slopes.
\end{abstract}

\maketitle

\section{Introduction}\label{sec:intro}
Let \(M\subset\mathbb R^2\) be an open subset of the symplectic plane
with the canonical symplectic form \(\omega=dx\wedge dy\).
A pair \(\mathcal W=(\mathcal F_1,\mathcal F_2)\) of Lagrangian
foliations of codimension \(1\) on \(M\), whose leaves are transverse to
each other, is called a \textit{Lagrangian \(2\)-web} on \(M\).
We regard such a web as an ordered pair of foliations. Two Lagrangian
\(2\)-webs are \textit{locally equivalent} if a local symplectic
diffeomorphism maps each foliation to the corresponding foliation.

Consider the four areas \(A,B,C,D\), measured with respect to \(\omega\),
cut out by the leaves as in Figure~\ref{fig:samuelson}.
The web satisfies the \textit{Samuelson condition} if
\[
\frac{A}{B}=\frac{C}{D}
\]
holds for every sufficiently small such configuration. We call it a
\textit{Samuelson web}. The standard Samuelson web is given by
\(x=\mathrm{const}\) and \(y=\mathrm{const}\), or by their restrictions
to an open subset of \(\mathbb R^2\).

\begin{figure}[t]
\centering
\includegraphics[width=0.48\textwidth]{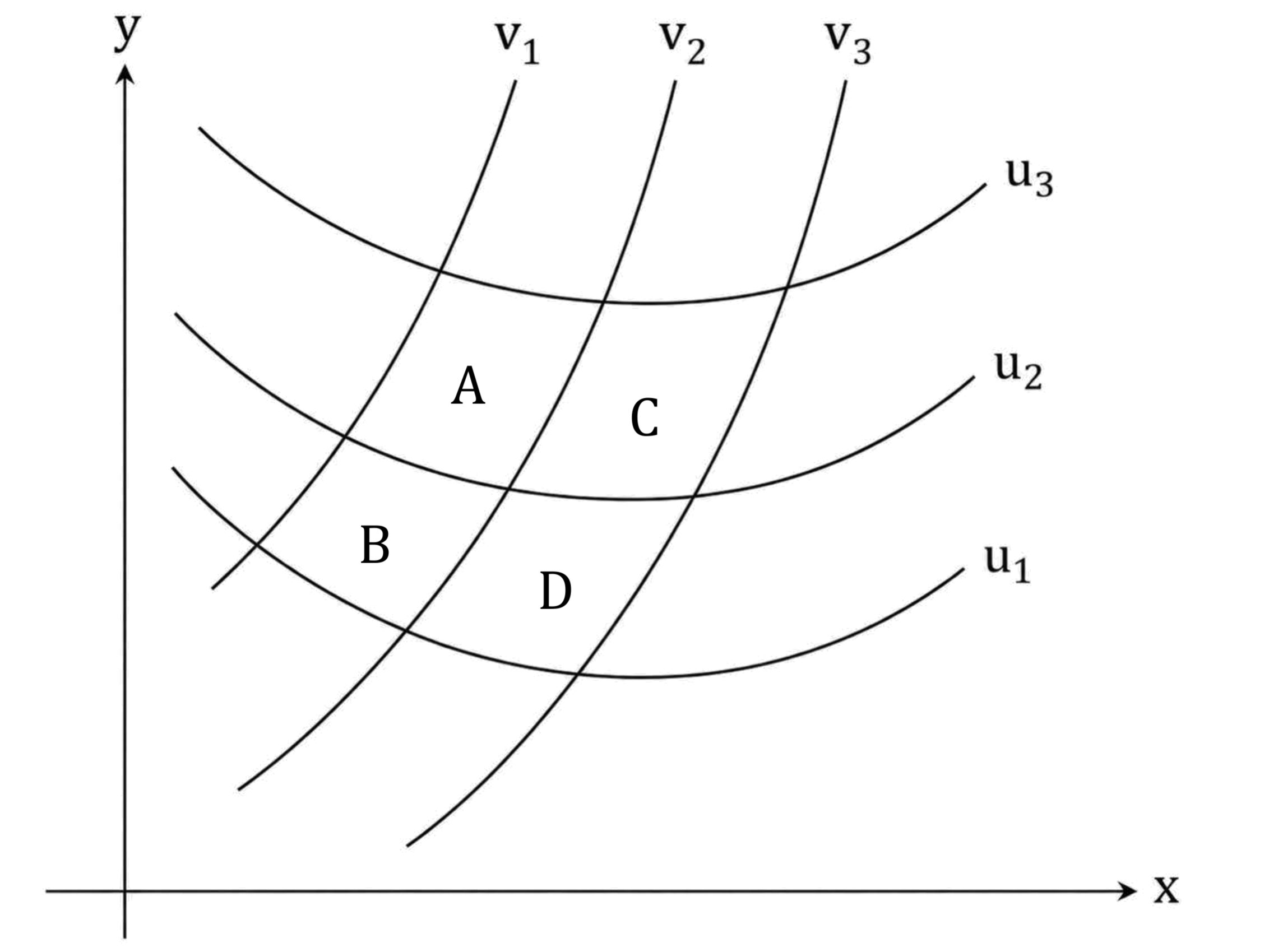}
\caption{The Samuelson area condition}
\label{fig:samuelson}
\end{figure}

Cooper, Russell, and Samuelson related this area condition to
thermodynamics and economic comparative statics
\cite{CooperRussellSamuelson}; see also
\cite{CooperRussell2006,CooperRussell2009,FerraraUdriste,Russell2011}.
The canonical torsion-free symplectic connection preserving both
foliations is called the Hess connection. It is flat exactly when
the web is locally symplectically equivalent to the standard web
\cite{Hess,Tabachnikov}. In the plane this is equivalent to the
Samuelson condition.

In a different setting, Goldberg and Lychagin studied Samuelson's
webs as planar \(4\)-webs. Their work focuses on ranks and differential
invariants \cite{GoldbergLychaginSamuelsonWebs}.

Domitrz and Zubilewicz \cite{DomitrzZubilewicz2023} studied local
invariants of divergence-free webs, that is, webs equipped with a
volume form. Their criteria for local triviality include the planar
Samuelson criterion recalled in Section~\ref{sec:samuelson}.

Tabachnikov asked which Lagrangian \(2\)-webs formed by affine tangent
spaces to Lagrangian submanifolds are locally trivial
\cite[p.~273]{Tabachnikov}. For comparison, the classical Graf--Sauer
theorem characterizes hexagonal planar \(3\)-webs of straight lines
by a degree-three curve in the dual plane
\cite{BlaschkeBol,Chern,PereiraPirio}.
For planar tangent \(2\)-webs, Domitrz and Zubilewicz
\cite[Theorem~23]{DomitrzZubilewicz2024} proved nonflatness under
natural regularity assumptions. Their theorem applies to smooth
curves and includes separated tangent families.

The Jacobian formula for the tangent-line intersection map is known
\cite[eq.~(4.54)]{DomitrzZubilewicz2024}. Here we study its behavior as
the two tangents coalesce. For the tangent line family \(L_t:y=tx+h(t)\),
let \(\Psi(u,v)=L_u\cap L_v\). In these coordinates, the Samuelson
condition holds precisely when
\(S_h=\partial_u\partial_v\log|J_\Psi|\) vanishes identically. Near a point where
\(h''\neq0\), we obtain
\[
J_\Psi(u,u+d)=dQ(u,d),\qquad
Q(u,0)=-\tfrac14 h''(u)^2,\qquad d=v-u,
\]
with \(Q\) smooth. Thus the Jacobian has a simple zero when the two
tangents coalesce. Theorem~\ref{thm:diagonal} gives the resulting expansion
\[
S_h(u,v)=\frac{1}{(v-u)^2}+R_h(u,v),
\]
where \(R_h\) is smooth across the diagonal. It also gives a positive
lower bound for \(S_h\) near the diagonal. The coefficient \(1\) is
independent of \(h\) in the common slope coordinates; this is the
universal obstruction considered here. The function \(S_h\) is also
a coefficient of the Hess curvature.
The conic computations in \cite{KurokawaConics} showed that the tangent
webs of non-degenerate real conics do not satisfy the Samuelson condition.
The expansion above explains this failure near the envelope without
relying on the algebraic form of conics.

For separated families, Proposition~\ref{prop:smooth-separated} gives
a local proof for two independently given smooth families
\(L_u:y=ux+f(u)\) and \(M_v:y=vx+g(v)\), with slopes in disjoint intervals.
The Samuelson condition would allow us to normalize the Jacobian of
their intersection map to one. Lemma~\ref{lem:rectilinear-area-rigidity}
then forces one of the two line families to be parallel, contradicting
the variation of both slopes.

All webs and equivalences are local. Sections~\ref{sec:samuelson}
and~\ref{sec:tangent} give the Jacobian criterion and the tangent
coordinates. The main diagonal expansion is Theorem~\ref{thm:diagonal}
in Section~\ref{sec:diagonal}. Section~\ref{sec:separated} proves
Proposition~\ref{prop:smooth-separated} for separated families.

\section{Samuelson condition and the Hess connection}\label{sec:samuelson}
Let \(u=u(x,y)\) and \(v=v(x,y)\) be smooth first integrals of an
ordered \(2\)-web on \(M\), with \(du\wedge dv\neq0\). The
\textit{web coordinate map}
\[
\phi:M\longrightarrow\mathbb R^2,\qquad \phi(x,y)=(u(x,y),v(x,y)),
\]
is a local diffeomorphism, but need not preserve area. On a sufficiently
small coordinate neighborhood let \(F=\phi^{-1}\). Then
\[
F^*\omega=J_F(u,v)\,du\wedge dv,\qquad
J_F(u,v)=\frac{1}{(u_xv_y-u_yv_x)(F(u,v))}.
\]
The sign of \(J_F\) is constant on a connected coordinate neighborhood,
whereas the density for unsigned areas is \(|J_F|\).

A change \((u,v)\mapsto(U(u),V(v))\), with \(U'V'\neq0\), is called
a \textit{recalibration}. Neither function is required to be increasing;
the two foliations remain ordered.

The following criterion is classical
\cite{CooperRussellSamuelson,FerraraUdriste,Tabachnikov}.
For the mixed-derivative criterion in the general divergence-free
setting, see \cite[Theorem~1]{DomitrzZubilewicz2023}.
\begin{proposition}\label{prop:samuelson-characterization}
On a sufficiently small rectangular web-coordinate neighborhood, the
following conditions are equivalent:
\begin{enumerate}
\item The web satisfies the Samuelson condition.
\item The unsigned area density is multiplicatively separable:
\[
|J_F(u,v)|=a(u)b(v),\qquad a(u)>0,\quad b(v)>0.
\]
\item The mixed logarithmic derivative vanishes identically:
\[
S_F(u,v):=\partial_u\partial_v\log|J_F(u,v)|=0.
\]
\item There is a recalibration \(R(u,v)=(U(u),V(v))\) such that
\(J_{R\circ\phi}=1\).
\end{enumerate}
\end{proposition}

The use of \(|J_F|\) in (2) is equivalent to a signed factorization
\(J_F=\widetilde a(u)\widetilde b(v)\) by nonvanishing functions.
Indeed, one can absorb the constant sign into either factor.
The equivalence of (2) and (3) follows by integrating
\(\partial_u\partial_v\log|J_F|=0\) on a rectangle. If (2) holds and
\(\varepsilon=\operatorname{sgn}J_F\), choose
\(U'=\varepsilon a\) and \(V'=b\) to obtain (4).
Conversely, (4) gives \(J_F=U'V'\).

The Hess connection in these coordinates is
\begin{equation}\label{eq:Hess}
\begin{aligned}
\nabla^H_{\partial_u}\partial_u
 &= (\partial_u\log|J_F|)\partial_u,&
\nabla^H_{\partial_v}\partial_v
 &= (\partial_v\log|J_F|)\partial_v,\\
\nabla^H_{\partial_u}\partial_v&=0,&
\nabla^H_{\partial_v}\partial_u&=0.
\end{aligned}
\end{equation}
These formulas follow from torsion-freeness, preservation of the
foliations, and \(\nabla^H\omega=0\) \cite{Hess,Tabachnikov}. With the convention
\[
R^H(X,Y)Z=\nabla^H_X\nabla^H_YZ-\nabla^H_Y\nabla^H_XZ-\nabla^H_{[X,Y]}Z,
\]
formula~\eqref{eq:Hess} gives
\begin{equation}\label{eq:Hess-curvature}
R^H(\partial_u,\partial_v)\partial_u=-S_F\partial_u,
\qquad
R^H(\partial_u,\partial_v)\partial_v=S_F\partial_v.
\end{equation}
Thus the Samuelson condition is equivalent to flatness of \(\nabla^H\).
Under a recalibration \(\widetilde u=U(u)\), \(\widetilde v=V(v)\),
the coefficient transforms as
\[
\widetilde S_F(\widetilde u,\widetilde v)
=\frac{S_F(u,v)}{U'(u)V'(v)}.
\]
In particular, its vanishing is intrinsic, although its numerical value
is coordinate-dependent.

\section{Tangent coordinates and the intersection map}\label{sec:tangent}
Write a smooth line family as
\[
L_t:\quad y=tx+h(t),
\]
where \(h\) is smooth on an open interval \(U\subset\mathbb R\).
At points where \(h''(t)\neq0\), this line family has a regular envelope
\[
\gamma(t)=(-h'(t),h(t)-th'(t)),
\]
and \(L_t\) is tangent to this envelope.

For \(u\ne v\), let
\[
\Psi(u,v)=(x(u,v),y(u,v))
\]
be the intersection point of \(L_u\) and \(L_v\).
The Jacobian formula below is the slope-coordinate form of
\cite[eq.~(4.54)]{DomitrzZubilewicz2024}.

\begin{lemma}\label{lem:intersection-map}
For \(u\ne v\), the intersection point is given by
\[
x(u,v)=\frac{h(v)-h(u)}{u-v},\;
y(u,v)=\frac{u h(v)-v h(u)}{u-v}.
\]
Moreover, set
\[
A(u,v)=h(v)-h(u)-h'(u)(v-u)
\]
and
\[
B(u,v)=h(u)-h(v)+h'(v)(v-u).
\]
Then the Jacobian of \(\Psi\) is
\begin{equation}\label{eq:J-tangent}
J_\Psi(u,v)=-\frac{A(u,v)B(u,v)}{(v-u)^3}.
\end{equation}
Finally, putting \(d=v-u\), one has
\[
J_\Psi(u,u+d)=dQ(u,d),
\]
where \(Q\) is smooth and
$
Q(u,0)=-\frac14 h''(u)^2.
$
\end{lemma}

\begin{proof}
Solving \(ux+h(u)=vx+h(v)\) gives the two coordinates.
Differentiation gives
\[
\begin{aligned}
x_u&=-\frac{A}{(v-u)^2},& x_v&=-\frac{B}{(v-u)^2},\\
y_u&=-\frac{vA}{(v-u)^2},& y_v&=-\frac{uB}{(v-u)^2}.
\end{aligned}
\]
Hence
\[
J_\Psi=x_u y_v-x_v y_u=-\frac{AB}{(v-u)^3}.
\]

Put \(d=v-u\), and work on the open set
\(\mathcal D=\{(u,d):u,u+d\in U\}\).
Taylor's formula with integral remainder gives the following explicit Hadamard factorizations:
\[
A(u,u+d)=d^2\alpha(u,d),\qquad
B(u,u+d)=d^2\beta(u,d),
\]
where
\[
\begin{aligned}
\alpha(u,d)&=\int_0^1(1-s)h''(u+sd)\,ds,\\
\beta(u,d)&=\int_0^1s\,h''(u+sd)\,ds.
\end{aligned}
\]
These functions are smooth jointly in \((u,d)\) on \(\mathcal D\).
At \(d=0\), the integral formulas give
\[
\alpha(u,0)=\beta(u,0)=\tfrac12h''(u).
\]
Consequently,
\[
J_\Psi(u,u+d)=dQ(u,d),\qquad Q(u,d)=-\alpha(u,d)\beta(u,d),
\]
where \(Q\) is smooth and \(Q(u,0)=-\tfrac14h''(u)^2\).
\end{proof}

\Needspace{5\baselineskip}
We next use this factorization to extend \(\Psi\) across the diagonal
and to obtain web coordinates on each ordered branch.

\begin{lemma}\label{lem:branches}
The map \(\Psi\) extends smoothly to \(U\times U\), with
\(\Psi(t,t)=\gamma(t)\). If \(h''\) is nowhere zero on an interval
\(V\subset U\), the restriction of \(\Psi\) to either of the sets
\[
T_+=\{(u,v)\in V^2:u<v\},\qquad
T_-=\{(u,v)\in V^2:u>v\}
\]
is a diffeomorphism onto its open image. On either chosen branch it is
the inverse web coordinate map of the corresponding tangent \(2\)-web.
\end{lemma}

\begin{proof}
With \(d=v-u\), the smooth factorization in
Lemma~\ref{lem:intersection-map} gives
\[
x(u,v)=-h'(u)-d\alpha(u,d),\qquad
 y(u,v)=u\,x(u,v)+h(u).
\]
These formulas extend \(\Psi\) smoothly to \(U\times U\) and give
\(\Psi(t,t)=\gamma(t)\). If \(h''\) is nowhere zero on \(V\),
its sign is constant. The integral formulas in
Lemma~\ref{lem:intersection-map} then show that \(\alpha(u,d)\) and
\(\beta(u,d)\) have that same nonzero sign whenever \(u,u+d\in V\).
Hence
\eqref{eq:J-tangent} gives \(J_\Psi\neq0\) on \(T_+\) and \(T_-\).
For a point \((x,y)\), the tangent parameters
are roots of
\[
h(t)+tx-y=0.
\]
The left-hand side is strictly convex or strictly concave on \(V\),
and therefore has at most two distinct roots. Ordering these roots
shows that \(\Psi\) is injective on each of \(T_+\) and \(T_-\).
An injective local diffeomorphism is a diffeomorphism onto its open
image. The curves \(u=\mathrm{const}\) and \(v=\mathrm{const}\) map
to parts of \(L_u\) and \(L_v\), respectively.
\end{proof}

Although \(\Psi(u,v)=\Psi(v,u)\), the choice of one ordering fixes the
two first integrals. On the diagonal its Jacobian vanishes, so the
smooth extension there is not a web coordinate chart.

\begin{remark}
\label{rem:DZ-comparison}
To compare our assumptions with those of Domitrz and Zubilewicz, put
\(C(u,v)=\gamma(v)-\gamma(u)\) and \(d=v-u\). Direct calculation gives
\begin{align*}
\det(\gamma'(u),\gamma'(v))&=h''(u)h''(v)d,\\
\det(\gamma'(u),C(u,v))&=h''(u)B(u,v),\\
\det(\gamma'(v),C(u,v))&=-h''(v)A(u,v),\\
\det(\gamma'(t),\gamma''(t))&=h''(t)^2.
\end{align*}
If \(h''(u)h''(v)A(u,v)B(u,v)d\neq0\), the tangents are
transverse, the chord is transverse at both contact points, and both
curvatures are nonzero. These are the assumptions in
\cite[Lemma~22]{DomitrzZubilewicz2024}. They hold for sufficiently
close distinct parameters near \(h''(t_0)\neq0\), so the known
nonflatness theorem applies.
\end{remark}

\Needspace{7\baselineskip}
We also record the mixed logarithmic derivative away from the diagonal.
The formula is the slope-coordinate form of
\cite[eq.~(4.57)]{DomitrzZubilewicz2024}.

\begin{lemma}\label{lem:Sh-formula}
For \(u\ne v\) and \(A(u,v)B(u,v)\ne0\), one has
\begin{equation}\label{eq:Sh-off-diagonal}
S_h(u,v)
:=
\duv\log|J_\Psi(u,v)|
=
h''(u)\frac{B(u,v)}{A(u,v)^2}
+
h''(v)\frac{A(u,v)}{B(u,v)^2}
-
\frac{3}{(v-u)^2}.
\end{equation}
\end{lemma}

\begin{proof}
By \eqref{eq:J-tangent},
\[
\log|J_\Psi|=\log|A|+\log|B|-3\log|v-u|.
\]
Using \(A_u=-(v-u)h''(u)\), \(A_v=h'(v)-h'(u)\), and
\(A+B=(v-u)(h'(v)-h'(u))\), we obtain
\[
\partial_u\partial_v\log|A|=h''(u)\frac{B}{A^2},\qquad
\partial_u\partial_v\log|B|=h''(v)\frac{A}{B^2}.
\]
The second equality follows by interchanging \(u\) and \(v\).
Since \(\partial_u\partial_v\log|v-u|=(v-u)^{-2}\), the formula follows.
\end{proof}

\section{The diagonal expansion near the envelope}\label{sec:diagonal}
The following theorem describes the singularity of the Hess curvature
coefficient \(S_h\) as the two tangent parameters approach one another.
The simple zero of \(J_\Psi\) gives a double pole with a smooth remainder.

\begin{theorem}\label{thm:diagonal}
Let \(h\in C^\infty(U)\), and suppose that \(h''(t_0)\neq0\).
There is an open interval \(V\subset U\) containing \(t_0\) and a
function \(R_h\in C^\infty(V\times V)\) such that
\begin{equation}\label{eq:pole}
S_h(u,v)=\frac{1}{(v-u)^2}+R_h(u,v),\qquad u,v\in V,\quad u\neq v.
\end{equation}
In particular,
\begin{equation}\label{eq:limit}
\lim_{\substack{(u,v)\to(t_0,t_0)\\u\neq v}}
(v-u)^2S_h(u,v)=1.
\end{equation}
After shrinking \(V\) further if necessary,
\begin{equation}\label{eq:positive}
S_h(u,v)\geq\frac{1}{2(v-u)^2}>0
\qquad (u,v\in V,\ u\neq v).
\end{equation}
\end{theorem}

\begin{proof}
By Lemma~\ref{lem:intersection-map}, on the open set
\(\{(t,d):t,t+d\in U\}\) the smoothly extended Jacobian has the form
\[
J_\Psi(t,t+d)=dQ(t,d),\qquad Q(t,0)=-\tfrac14h''(t)^2.
\]
Since \(Q(t_0,0)\neq0\), choose \(V\) so that
\(Q(u,v-u)\neq0\) throughout \(V\times V\). For \(u\neq v\),
\[
\log|J_\Psi(u,v)|=\log|v-u|+G_h(u,v),\qquad
G_h(u,v):=\log|Q(u,v-u)|.
\]
The function \(G_h\) is smooth on the full open square \(V\times V\),
including its diagonal. Define
\begin{equation}\label{eq:remainder}
R_h(u,v):=\partial_u\partial_vG_h(u,v)
\qquad ((u,v)\in V\times V).
\end{equation}
Here both partial derivatives are taken after composing \(Q\) with
\((u,v)\mapsto(u,v-u)\). Since
\[
\partial_u\partial_v\log|v-u|=\frac{1}{(v-u)^2},
\]
we obtain \eqref{eq:pole} and \eqref{eq:limit}.

On a smaller square with compact closure in the preceding one,
\(|R_h|\leq C\) for some constant \(C\geq0\). Shrink \(V\) again so
that \(C(\operatorname{diam}V)^2\leq\tfrac12\). Then
\[
S_h(u,v)\geq\frac{1}{(v-u)^2}-C
\geq\frac{1}{2(v-u)^2}.
\]
\end{proof}

\begin{remark}\label{rem:meaning}
The coefficient \(1\) in \eqref{eq:pole} comes from the simple zero of
\(J_\Psi\) along \(u=v\); it is independent of \(h\) in the common slope
coordinates. Although \(R_h\) extends smoothly, the diagonal is not
part of the web because \(J_\Psi=0\) there.
\end{remark}

We now apply Theorem~\ref{thm:diagonal} to the tangent web of a smooth
plane curve with nonzero curvature.

\begin{corollary}\label{cor:curve}
Let \(\gamma:I\to\mathbb R^2\) be a \(C^\infty\) regular plane curve,
and let \(p=\gamma(r_0)\) have nonzero curvature. The tangent \(2\)-web
formed by sufficiently nearby distinct tangents near \(p\), on either
ordered branch, does not satisfy the Samuelson condition.
\end{corollary}

\begin{proof}
Translations and orientation-preserving rotations preserve the
symplectic form. After such a change of coordinates, write the curve
locally as \(Y=f(X)\), with \(f''(x_0)\neq0\). Put \(t=f'(X)\).
The inverse function theorem gives \(X=x(t)\), and the tangent line is
\[
L_t:\quad Y=tX+h(t),\qquad h(t)=f(x(t))-t\,x(t).
\]
Differentiation yields
\[
h'(t)=-x(t),\qquad h''(t)=-x'(t)=-\frac{1}{f''(x(t))}\neq0.
\]
Lemma~\ref{lem:branches} gives web coordinates on either ordered
branch. The positive lower bound in Theorem~\ref{thm:diagonal} and
Proposition~\ref{prop:samuelson-characterization} then imply the
non-Samuelson conclusion.
\end{proof}

\section{A smooth local proof for separated families}
\label{sec:separated}
We first consider area-preserving maps that send both families of
coordinate lines to straight lines. The resulting lemma will be
applied to a recalibrated intersection map.

\Needspace{10\baselineskip}
\begin{lemma}\label{lem:rectilinear-area-rigidity}
Let $P,Q\subset\R$ be nonempty open intervals and let
$F:P\times Q\to\R^2$ be a smooth map with
\[
\det(F_s,F_t)=1.
\]
Suppose that, for each fixed $t_0\in Q$ and $s_0\in P$, the curves
\[
s\longmapsto F(s,t_0),\qquad t\longmapsto F(s_0,t)
\]
are each contained in a straight line. Then at least one of these
two families consists of mutually parallel lines.
\end{lemma}

\begin{proof}
Since $\det(F_s,F_t)=1$, the vectors $F_s,F_t$ are a basis at each
point. Straightness of the coordinate curves gives smooth functions
$\alpha,\beta$ such that
\begin{equation}\label{eq:straight-pure}
F_{ss}=\alpha F_s,\qquad F_{tt}=\beta F_t.
\end{equation}
Differentiating $\det(F_s,F_t)=1$ in $s$ and $t$ gives
\begin{equation}\label{eq:straight-mixed}
F_{st}=-\beta F_s-\alpha F_t.
\end{equation}
Indeed, the $F_t$ coefficient of $F_{st}$ is $-\alpha$ by the
$s$ derivative of the determinant, and its $F_s$ coefficient is
$-\beta$ by the $t$ derivative.

The mixed derivatives satisfy
\begin{align*}
F_{sst}&=(\alpha_t-\alpha\beta)F_s-\alpha^2F_t,\\
F_{sts}&=-\beta_sF_s+(\alpha^2-\alpha_s)F_t,\\
F_{stt}&=(\beta^2-\beta_t)F_s-\alpha_tF_t,\\
F_{tts}&=-\beta^2F_s+(\beta_s-\alpha\beta)F_t.
\end{align*}
Comparing coefficients in $F_{sst}=F_{sts}$ and $F_{stt}=F_{tts}$ gives
\begin{equation}\label{eq:straight-compatibility}
\alpha_s=2\alpha^2,\qquad
\beta_t=2\beta^2,\qquad
\alpha_t+\beta_s=\alpha\beta.
\end{equation}

Put $w=\alpha_t$. Equation~\eqref{eq:straight-compatibility} and equality of
mixed derivatives give
\begin{equation}\label{eq:straight-w-system}
\beta_s=\alpha\beta-w,\qquad
w_s=4\alpha w,\qquad
w_t=5\beta w-2\alpha\beta^2.
\end{equation}
The last identity follows by comparing the $t$ derivative of
$\beta_s=\alpha\beta-w$ with the $s$ derivative of
$\beta_t=2\beta^2$. A further mixed-derivative comparison gives
\begin{align*}
(w_s)_t&=4w^2+20\alpha\beta w-8\alpha^2\beta^2,\\
(w_t)_s&=29\alpha\beta w-5w^2-8\alpha^2\beta^2.
\end{align*}
Consequently,
\begin{equation}\label{eq:straight-w-product}
w(w-\alpha\beta)=0.
\end{equation}

We claim that $\alpha\beta=0$ everywhere. Otherwise, choose a connected
neighborhood on which $\alpha\beta\neq0$. By
\eqref{eq:straight-w-product}, the continuous function $w/(\alpha\beta)$
takes only the values $0$ and $1$, hence is constant there.
If $w=0$, the last equation of \eqref{eq:straight-w-system} gives
$0=-2\alpha\beta^2$, a contradiction. If $w=\alpha\beta$, then
$\beta_s=0$, so
\[
w_s=(\alpha\beta)_s=2\alpha^2\beta,
\]
whereas \eqref{eq:straight-w-system} gives $w_s=4\alpha^2\beta$, again a
contradiction. This proves the claim.

Now \eqref{eq:straight-w-product} gives $w=0$, so
$\alpha_t=\beta_s=0$. Thus $\alpha=\alpha(s)$ and $\beta=\beta(t)$.
If $\alpha(s_0)\neq0$ at some $s_0\in P$, then
$\alpha(s_0)\beta(t)=0$ for every $t\in Q$, so $\beta\equiv0$.
Otherwise $\alpha\equiv0$.

If $\alpha\equiv0$, then $F_{ss}=0$ and
$F_{st}=-\beta(t)F_s$. Hence $F_s$ is independent of $s$ and varies with $t$ by a nonzero
scalar factor. Its direction is therefore fixed, so the
$s$-coordinate curves are mutually parallel. If $\beta\equiv0$, the same argument shows that the
$t$-coordinate curves are mutually parallel.
\end{proof}

\Needspace{15\baselineskip}
We now combine the lemma with the Samuelson factorization to treat
two line families with varying slopes.

\begin{proposition}\label{prop:smooth-separated}
Let $I,J\subset\R$ be nonempty disjoint open intervals, and let
$f\in C^\infty(I)$ and $g\in C^\infty(J)$. Consider the two line
families
\[
L_u:\ y=ux+f(u),\qquad
M_v:\ y=vx+g(v).
\]
Let $\Psi(u,v)=L_u\cap M_v$, and let $D\subset I\times J$ be an
open set on which $J_\Psi=\det D\Psi$ is nonzero. Then
\[
S(u,v)=\partial_u\partial_v\log|J_\Psi(u,v)|
\]
cannot vanish identically on any nonempty open subset of $D$.
Equivalently, no local $2$-web obtained from an inverse branch of
$\Psi$ on $D$ satisfies the Samuelson condition for $dx\wedge dy$.
\end{proposition}

\begin{proof}
Suppose that $S=0$ on a nonempty open subset of $D$. Choose a
nonempty open rectangle $I_0\times J_0$ within that subset. The vanishing mixed derivative
implies, as in Proposition~\ref{prop:samuelson-characterization},
\begin{equation}\label{eq:separated-factorization}
J_\Psi(u,v)=a(u)b(v)
\end{equation}
for smooth nonvanishing functions $a,b$, with the constant sign of
$J_\Psi$ absorbed into one factor. Fix $u_0\in I_0$ and $v_0\in J_0$, and set
\[
s=\int_{u_0}^u a(\xi)\,d\xi,\qquad
t=\int_{v_0}^v b(\eta)\,d\eta.
\]
Since $a,b$ are nowhere zero, these maps are diffeomorphisms from
$I_0,J_0$ onto open intervals $P,Q$, respectively. Denote their
inverses by $u=u(s)$ and $v=v(t)$. The map
\[
F(s,t)=\Psi(u(s),v(t)),\qquad (s,t)\in P\times Q,
\]
has
\[
\det(F_s,F_t)=J_\Psi\,u_s v_t=1.
\]
For fixed $t$, its image is contained in $M_{v(t)}$, and for fixed
$s$, its image is contained in $L_{u(s)}$. Thus both coordinate
families are straight. Lemma~\ref{lem:rectilinear-area-rigidity}
implies that at least one of these families is parallel. This is
impossible: their respective slopes are $v(t)$ and $u(s)$, and
\[
v_t=\frac1{b(v(t))}\neq0,\qquad
u_s=\frac1{a(u(s))}\neq0.
\]
This contradiction proves the proposition.
\end{proof}

Proposition~\ref{prop:smooth-separated} does not require \(f''\) or
\(g''\) to be nonzero, so the envelopes need not be regular.

\Needspace{11\baselineskip}
Taking both families from a single smooth function \(h\) gives the
following corollary.

\begin{corollary}
\label{prop:separated}
Let \(h\in C^\infty(U)\), where \(U\) is an open interval, and let
\(I,J\subset U\) be nonempty disjoint open intervals. Let
\(\Psi(u,v)=L_u\cap L_v\) for \(L_t:y=tx+h(t)\), and let
\(D\subset I\times J\) be open with \(J_\Psi\neq0\) on \(D\).
Then \(S_h\) cannot vanish identically on any nonempty open subset of
\(D\). In particular, no local tangent \(2\)-web defined by an inverse
branch of \(\Psi\) on \(D\) satisfies the Samuelson condition.
\end{corollary}

\begin{proof}
Apply Proposition~\ref{prop:smooth-separated} with
\(f=h|_I\) and \(g=h|_J\).
\end{proof}

Only \(h|_I\) and \(h|_J\) enter the proof. Unlike the estimate in
Theorem~\ref{thm:diagonal}, this result does not assert that
\(S_h(u,v)\neq0\) at every separated pair; it excludes vanishing on
any nonempty open set.

\begin{remark}\label{rem:parallel-exception}
Lemma~\ref{lem:rectilinear-area-rigidity} does not exclude all straight-line
Samuelson webs. For example,
\[
F(s,t)=\bigl(\sqrt{2s},\ t\sqrt{2s}\bigr),\qquad s>0,
\]
has Jacobian one. Its \(s\)-coordinate curves lie on lines through the
origin, whereas its \(t\)-coordinate curves lie on vertical parallel
lines. The contradiction above requires that both slope parameters vary.
\end{remark}

We now compare the proof of Proposition~\ref{prop:smooth-separated}
with that of Domitrz and Zubilewicz.

\begin{remark}\label{rem:smooth-DZ-method}
Domitrz and Zubilewicz \cite[proof of Theorem~23]{DomitrzZubilewicz2024}
write the two curve germs as graphs and assume that the coefficient
in their curvature formula vanishes identically. They differentiate
this identity and eliminate derivatives of the graph functions to
obtain a contradiction. Their restriction to equal graph parameters still
compares distinct points on the two curves; it is not the
coalescing-tangent limit of Section~\ref{sec:diagonal}.

In the proof of Proposition~\ref{prop:smooth-separated}, the Samuelson
factorization first gives coordinates with \(\det(F_s,F_t)=1\).
The straight coordinate curves give equations for the second
derivatives of \(F\). Comparing mixed derivatives forces one family
to be parallel, by Lemma~\ref{lem:rectilinear-area-rigidity}, contrary
to the variation of both slopes. Both proofs use local differential
identities, but their unknown functions and equations are different.
\end{remark}

\Needspace{8\baselineskip}

\end{document}